\documentclass[11pt,a4paper,final,reqno]{amsart}
 
\usepackage[notcite,notref,color]{showkeys} 

\usepackage{amsmath,amssymb,amsthm,bbm} 
\usepackage{enumitem} 
\usepackage[overload]{textcase}
 
\usepackage{epsfig,psfrag,color} 

\usepackage{fullpage}

\usepackage{tikz}
\usetikzlibrary{calc}
\usetikzlibrary{shapes,arrows,positioning,decorations.pathreplacing,calc}

\usepackage[colorlinks,bookmarks,linkcolor=black,citecolor=black]{hyperref} 
 
\usepackage[english,algoruled,lined,noresetcount,norelsize]{algorithm2e}

\setlist{itemsep=1pt,parsep=0pt,topsep=2pt,partopsep=0pt}  
\setenumerate{leftmargin=*} 

\def\itm#1{\rm ({#1})} 
\def\itmit#1{\itm{\it #1\,}} 
\def\rom{\itmit{\roman{*}}}

\def\endofFact{\hfill\scalebox{.6}{$\Box$}}

\SetTitleSty{textbf}{}
\SetArgSty{textrm}

\SetKw{Failure}{failure}
\SetKw{Return}{return}
\SetKw{NOT}{not}
\SetKwFor{RepeatForever}{repeat}{forever}{end}

\let\subset\subseteq  
\let\eps\varepsilon 
\let\rho\varrho 
  
\newcommand{\rev}[1]{\overleftarrow{#1}}

\def\cE{\mathcal{E}}
\def\cF{\mathcal{F}}
\def\cG{\mathcal{G}}
\def\cH{\mathcal{H}}
\def\cP{\mathcal{P}}

\def\PP{\mathbb{P}}
\def\EE{\mathbb{E}}

\newtheorem{theorem}{Theorem}
\newtheorem{lemma}[theorem] {Lemma}

\newtheorem{definition}[theorem] {Definition}  
  
\newtheorem{claim}[theorem] {Claim}  

\theoremstyle{remark}

\newcommand{\oldqed}{}
\newenvironment{claimproof}[1][Proof]{
  \renewcommand{\oldqed}{\qedsymbol}
  \renewcommand{\qedsymbol}{\endofFact}
  \begin{proof}[#1]
}{
  \end{proof}
  \renewcommand{\qedsymbol}{\oldqed}
} 

\newcommand{\bigO}{\mathcal{O}}
\newcommand{\Prob}{\mathbb{P}}
\newcommand{\Exp}{\mathbb{E}}

\newcommand{\mult}{\text{mult}}

\newcommand{\Gr}[1][r]{\cG^{(#1)}}

\newcommand{\Pres}{P_\mathrm{res}}
\newcommand{\Palm}{P_\mathrm{almost}}

\newcommand{\EMAIL}[1]{  \textit{E-mail}: \texttt{#1} }

\newcommand{\tpl}[1]{\mathbf{#1}}

\title{Finding tight Hamilton cycles in random hypergraphs faster}

  \author[P. Allen]{Peter Allen*}

  \thanks{
    *
    Department of Mathematics, London School of Economics, Houghton Street,
London WC2A 2AE, U.\ K.\
   \EMAIL{p.d.allen@lse.ac.uk}
 }
  \author[C. Koch]{Christoph Koch\dag}
\thanks{
    \dag 
    Mathematics Institute, University of Warwick, Zeeman Building, Coventry CV4~7AL, U.\ K.\
   \EMAIL{c.koch@warwick.ac.uk}
 }
  \author[O. Parczyk]{Olaf Parczyk\ddag}
  \author[Y. Person]{Yury Person\ddag}
  \thanks{
    \ddag  
Institut f\"ur Mathematik, Goethe Universit\"at, Robert-Mayer-Str. 6-10 60325 Frankfurt am Main, Germany. 
   \EMAIL{parczyk|person@math.uni-frankfurt.de}
  }

  \thanks{
    The second author was supported by Austrian Science Fund (FWF): P26826, and European Research Council (ERC): No.~639046.
    The third and fourth authors were supported by DFG grant PE 2299/1-1.
  }

\date{\today}

\begin{document}
\begin{abstract}
 In an $r$-uniform hypergraph on $n$ vertices a tight Hamilton cycle consists of $n$ edges such that there exists a cyclic ordering of the vertices where the edges correspond to consecutive segments of $r$ vertices.
 We provide a first deterministic polynomial time algorithm, which finds a.a.s.\ tight Hamilton cycles in random $r$-uniform hypergraphs with edge probability at least $C \log^3n/n$.
 
 Our result partially answers a question of Dudek and Frieze [Random Structures \& Algorithms  42 (2013), 374--385] who proved that tight Hamilton cycles exists already for $p=\omega(1/n)$ for $r=3$ and $p=(e + o(1))/n$ for $r\ge 4$ using a second moment argument.
 Moreover our algorithm is superior to previous results of Allen, B\"ottcher, Kohayakawa and Person [Random Structures \& Algorithms 46 (2015), 446--465] and Nenadov and {\v{S}}kori{\'c} [arXiv:1601.04034] in various ways: the algorithm of Allen et al.\ is a randomised polynomial time algorithm working for edge probabilities $p\ge n^{-1+\eps}$, while the algorithm of Nenadov and {\v{S}}kori{\'c} is a randomised quasipolynomial time algorithm working for edge probabilities $p\ge C\log^8n/n$.
\end{abstract}
\maketitle


\section{Introduction}\label{sec:intro}
The Hamilton Cycle Problem, i.e., deciding whether a given graph contains a Hamilton cycle, is one of the $21$ classical $\mathrm{NP}$-complete problems due to  Karp~\cite{karp1972reducibility}.
The best currently known algorithm is due to Bj\"orklund~\cite{bjorklund2010determinant}: a Monte-Carlo algorithm with worst case running time $\bigO^*(1.657^n)$,\footnote{Writing $\bigO^*$ means we ignore polylogarithmic factors.} without false positives and false negatives occurring only with exponentially small probability.
But what about ``typical'' instances? In other words, when the input is a random graph sampled from some specific distribution, is there an algorithm which finds a Hamilton cycle in polynomial time with small error probabilities?

For example, let us examine the classical binomial random graph $\cG(n,p)$: P\'osa~\cite{Posa} and Korshunov~\cite{Kor76,Kor77} proved that the hamiltonicity threshold is at $p=\Theta(\log n/n)$.
Their result was improved by Koml\'os and Szemer\'edi~\cite{KomSzem} who showed that the hamiltonicity threshold coincides with the threshold for minimum degree $2$, and Bollob\'as~\cite{Boll} demonstrated that this is even true for the hitting times of these two properties in the corresponding random graph process.
But these results do not allow one to actually find any Hamilton cycle in polynomial time.
The first polynomial time randomised algorithms for finding Hamilton cycles in $\cG(n,p)$ are due to Angluin and Valiant~\cite{AngVal} and Shamir~\cite{Shamir}.
Subsequently, Bollob\'as, Fenner and Frieze~\cite{BFF} developed a deterministic algorithm, whose success probability (for input sampled from $\cG(n,p)$) matches the probability of $\cG(n,p)$ being hamiltonian in the limit as $n\to\infty$.

Turning to hypergraphs, there exist various notions of Hamilton cycles: weak Hamilton cycle, Berge Hamilton cycle, $\ell$-overlapping Hamilton cycles (for $\ell\in[r-1]$).
In each situation, one seeks to cyclically order the vertex set such that:
\begin{itemize}
\item any two consecutive vertices lie in a hyperedge (a \emph{weak Hamilton cycle}),
\item any two consecutive vertices lie in some chosen hyperedge and no hyperedge is chosen twice (a \emph{Berge Hamilton cycle}),
\item the edges are consecutive segments so that two consecutive edges intersect in exactly $\ell$ vertices 
(an \emph{$\ell$-overlapping Hamilton cycle}).
\end{itemize}

The (binomial) random $r$-uniform hypergraph $\Gr[r](n,p)$ defined on the vertex set $[n]:=\{1,\dots, n\}$, includes each \emph{$r$-set} $x\in\binom{[n]}{r}$ as an \emph{(hyper-)edge} independently with probability $p=p(n)$.
The study of Hamilton cycles in random hypergraphs was initiated more recently by Frieze in~\cite{Frieze}, who considered so-called loose cycles in $3$-uniform hypergraphs (these are $1$-overlapping cycles in our terminology).
 Dudek and Frieze~\cite{DudFriLoose,DudFriTight} determined, for all $\ell$ and $r$, the threshold for the appearance of 
 an $\ell$-overlapping Hamilton cycle in a random $r$-uniform hypergraph (most thresholds being determined exactly, some only asymptotically).
 However, these results were highly nonconstructive, relying either on a 
 result of Johansson, Kahn and Vu~\cite{JKV} or the second moment method.

 The case of weak Hamilton cycles was studied by Poole in~\cite{poole2014weak}, while Berge Hamilton cycles in random hypergraphs were studied by Clemens, Ehrenm\"uller and Person in~\cite{CEP16}, the latter one being algorithmic.

In the case $\ell=r-1$ it is customary to refer to  an $\ell$-overlapping cycle as a \emph{tight} cycle.
Thus,  the tight $r$-uniform cycle on vertex set $[n]$, $n\ge r$, has edges $\{i + 1, . . . , i + r\}$ for all $i$, where we identify vertex $n+i$ with $i$.
A general result of Friedgut~\cite{Friedgut} readily shows that the threshold for the appearance of an $\ell$-overlapping cycle in $\Gr[r](n,p)$ is sharp; that is, there is some threshold function $p_0=p_0(n)$ such that for any constant $\eps>0$ the following holds.
If $p\le(1-\eps)p_0$ then $\Gr[r](n,p)$ a.a.s.\ does not contain the desired cycle, whereas if $p\ge(1+\eps)p_0$ then it a.a.s.\ does contain the desired cycle.
Dudek and Frieze~\cite{DudFriTight} proved that for $r\ge 4$ the function $p_0(n)=e/n$ is a threshold function for containment of a tight cycle, while for $r=3$ they showed that a.a.s.\ $\Gr[3](n,p)$  contains a tight Hamilton cycle for any $p=p(n)=\omega(1/n)$.
An easy first moment calculation shows that if $p=p(n)\le(1-\eps)e/n$ then a.a.s.\ $\Gr[r](n,p)$ does not contain a tight Hamilton cycle.

\subsection{Main result}
At the end of~\cite{DudFriTight}, Dudek and Frieze posed the question of finding algorithmically various $\ell$-overlapping Hamilton cycles at the respective thresholds.
In this paper we study tight Hamilton cycles and provide  a first deterministic polynomial time algorithm, which works for $p$ only slightly above the threshold.

\begin{theorem}\label{thm:main} 
  For each integer $r\ge 3$ there exists $C>0$ and a deterministic polynomial time algorithm  with runtime $O(n^r)$ which for any $p\ge C(\log n)^3n^{-1}$ a.a.s.\ finds
  a tight Hamilton cycle in the random $r$-uniform hypergraph
  $\cG^{(r)}(n,p)$.
\end{theorem}

Prior to our work there were two algorithms known that dealt with finding tight cycles.
The first algorithmic proof 
was given by B\"ottcher, Kohayakawa and the first and the fourth authors in~\cite{TightCycle}, where they presented a 
randomised polynomial time algorithm which could find tight cycles a.a.s.\ at the edge probability $p\ge  n^{-1+\eps}$ for any fixed $\eps\in (0,1/6r)$ and running time $n^{20/\eps^2}$. 
 The second result is a randomised quasipolynomial time algorithm of Nenadov and \v{S}kori\'{c}~\cite{nenadov2016powers}, which works for $p\ge C(\log n)^8/n$.

Our result builds on the adaptation of the absorbing technique of
 R\"odl, Ruci\'nski and Szemer\'edi~\cite{RRSz06} to sparse random (hyper-)graphs.
 This technique was actually used earlier 
 by Krivelevich in~\cite{Kri97} in the context of random graphs.
 However, the first results that provided essentially optimal thresholds (for other problems) are proved in~\cite{TightCycle} mentioned above in the context of random hypergraphs and independently 
by K\"uhn and Osthus in~\cite{KOPosa}, who studied the threshold for the appearance of powers of Hamilton cycles 
in random graphs.
The probability of  $p\ge C(\log n)^3n^{-1}$ results in the use of so-called reservoir 
structures of polylogarithmic size, as first used by 
Montgomery to find spanning trees in random graphs~\cite{M14b}, and later in~\cite{nenadov2016powers}.
Our improvements result in the combination of the two algorithmic approaches~\cite{TightCycle,nenadov2016powers} and in the analysis of a simpler algorithm that we provide.


\smallskip

\paragraph{\bf Organisation.}
In Section~\ref{sec:outline} we provide an informal overview of our algorithm.
In Section~\ref{sec:proof} we then provide two key lemmas and the proof of Theorem~\ref{thm:main} which rests on these lemmas.
In the subsequent 
sections we prove these main lemmas: the Connecting Lemma and the Reservoir Lemma.

\section{An informal algorithm overview}\label{sec:outline}
\subsection{Notation and inequalities}
An $s$-\emph{tuple} $(u_1,\dots,u_s)$ of vertices is an ordered set of distinct
vertices.
We often denote tuples by bold symbols, and occasionally also
omit the brackets and write $\tpl{u}=u_1,\dots,u_s$.
Additionally, we may
also use a tuple as a set and write for example, if~$S$ is a set,
$S\cup\tpl{u}:=S\cup\{u_i\colon i\in[s]\}$.
The \emph{reverse} of the
$s$-tuple $\tpl{u}$ is the $s$-tuple $\rev{\tpl{s}}:=(u_s,\dots,u_1)$.

In an $r$-uniform hypergraph~$\cG$ the tuple $P=(u_1,\dots,u_\ell)$ forms a
\emph{tight path} if the set $\{u_{i+1},\dots,u_{i+r}\}$ is an edge for every $0\le
i\le \ell-r$.
For any $s\in[\ell]$ we say that~$P$ \emph{starts} with the
$s$-tuple $(u_1,\dots,u_s)=:\tpl{v}$ and \emph{ends} with the $s$-tuple
$(u_{\ell-(s-1)},\dots,u_\ell)=:\tpl{w}$.
We also call~$\tpl{v}$ the
\emph{start $s$-tuple} of $P$, $\tpl{w}$ the \emph{end $s$-tuple} of
$P$, and~$P$ a $\tpl{v}-\tpl{w}$ path.
The \emph{interior} of~$P$ is formed
by all its vertices but its start and end $(r-1)$-tuples.
Note that the interior
of $P$ is not empty if and only if $\ell>2(r-1)$.

For a binomially distributed random variable $X$ and a constant $0< \gamma <1$ we will apply the following Chernoff-type bound (see, e.g.,  {\cite[Corollary~2.3]{JaLuRu:Book}})
\begin{align}
\label{eq:Chernoff}
\PP \left[ |X - \EE(X) | \le \gamma \EE(X) \right] &\le 2 \exp\left( - \frac{\gamma^2 \EE(X)}{3} \right).
\end{align}

In addition we will make use of the following consequence of Janson's inequality (see
for example~\cite{JaLuRu:Book}, Theorem 2.18
): Let~$\Omega$ be a finite set
and~$\cP$ be a family of non-empty subsets of~$\Omega$.
Now consider the
random experiment where each $e\in\Omega$ is chosen independently with
probability~$p$ and define for each~$P\in\cP$ the indicator variable $I_P$
that each element of~$P$ gets chosen.
Set~$X=\sum_{P\in\cP} I_P$ and
$\Delta= \sum_{P\neq P',P\cap
	P'\neq\emptyset}\Exp(I_PI_{P'})$.
Then
\begin{align}
\label{eq:Janson}
	\PP[X=0]\le\exp \left(- \frac{\EE (X)^2}{\EE (X) + \Delta} \right).
\end{align}

\subsection{Overview of the algorithm}
We start with the given sample of the random hypergraph $\Gr[r](n,p)$ and we will reveal the edges as we proceed.
First, using the Reservoir Lemma (Lemma~\ref{lem:respath} below), we construct a tight path $\Pres$ which covers a small but bounded away from zero fraction of $[n]$, which has the \emph{reservoir property}, namely that there is a set $R\subset V(\Pres)$ of size $2 C p^{-1} \log n \le 2n/\log^2 n$ such that for any $R'\subset R$, there is a tight path covering exactly the vertices $V(\Pres)\setminus R'$ whose ends are the same as those of $\Pres$, and this tight path can be found given $\Pres$ and $R'$ in time polynomial in $n$ a.a.s.

 We now greedily extend $\Pres$, choosing new vertices when possible and otherwise vertices in $R$.
 We claim that a.a.s.\ this strategy produces a structure $\Palm$ which is almost a tight path extending $\Pres$ and covering $[n]$.
 The reason it is only `almost' a tight path is that some vertices in $R$ may be used twice.
 We denote the set of vertices used twice by $R_1'$.
 But we will succeed in covering $[n]$ with high probability.
 Recall that, due to the reservoir property, we can dispense with the vertices from $R'_1$ in the part $\Pres$ of the almost tight Hamilton path $\Palm$.
 
 Finally, we apply the Connecting Lemma (Lemma~\ref{lem:connect} below) to find a tight path in $R\setminus R_1'$ joining the ends of $\Palm$, and using the reservoir property this gives the desired tight Hamilton cycle.

This approach is similar to that in~\cite{TightCycle}.
The main difference is the way we prove the Reservoir Lemma (Lemma~\ref{lem:respath}).
In both~\cite{TightCycle} and this paper, we first construct many small, identical, vertex-disjoint \emph{reservoir structures} (in some part of the literature, mostly in the dense case, this structure is called an absorber).
A reservoir structure contains a spanning tight path, and a second tight path with the same ends which omits one \emph{reservoir vertex}.
We then use Lemma~\ref{lem:connect} to join the ends of all these reservoir structures together into the desired $\Pres$.
In~\cite{TightCycle}, reservoir structures are of constant size (depending on the $\eps$) and they are found by using brute-force search.
This is slow, and is also the cause of the algorithm in~\cite{TightCycle} being randomised: there it is necessary to simulate exposure in rounds of the random hypergraph since the brute-force search reveals all edges.
In this paper, by contrast, we construct reservoir structures by a local search procedure which is both much faster and reveals much less of the random hypergraph.

We will perform all the constructions in this paper by using local search procedures.
At each step we reveal all the edges of $\cG^{(r)}(n,p)$ which include a specified $(r-1)$-set, the \emph{search base}.
The number of such edges will always be in expectation of the order of $pn$, so that by Chernoff's inequality and the union bound, with high probability at every step in the algorithm the number of revealed edges is close to the expected number.
Of course, what we may not do is attempt to reveal a given edge twice: we therefore keep track of an \emph{exposure hypergraph} $\cE$, which is the $(r-1)$-uniform hypergraph consisting of all the $(r-1)$-sets which have been used as search bases up to a given time in the algorithm.
We will show that $\cE$ remains quite sparse, which means that at each step we have almost as much freedom as at the start when no edges are exposed.

For concreteness, we use a doubly-linked list of vertices as the data structure representing a tight (almost-) path.
However this choice of data structure is not critical to the paper and we will not further comment on it.
The reader can easily verify that the various operations we describe can be implemented in the claimed time using this data structure.
To simplify readability, we will omit in the calculations floor and ceiling signs whenever they are not crucial for the arguments.

\section{Two key Lemmas and the proof of Theorem~\ref{thm:main}}
\label{sec:proof}
\subsection{Two Key Lemmas}
Recall the definition of the \emph{reservoir path}  $\Pres$.
It is an $r$-uniform hypergraph with a special subset $R\subsetneq V(\Pres)$ and some start and end $(r-1)$-tuples $\tpl{v}$ and $\tpl{w}$ respectively, such that:
\begin{enumerate}
\item $\Pres$ contains a tight path with the vertex set $V(\Pres)$ and the `end tuples' $\tpl{v}$ and $\tpl{w}$, and
\item for \emph{any} $R'\subseteq R$, $\Pres$ contains a tight path with the vertex set $V(\Pres)\setminus R'$ and the `end tuples' $\tpl{v}$ and $\tpl{w}$.
\end{enumerate}

We first give the lemma which constructs $\Pres$.
In addition to with high probability returning $\Pres$, we also need to describe the likely resulting exposure hypergraph.

\begin{lemma}[Reservoir Lemma]\label{lem:respath}
	For each $r\ge 3$ and $p \in (0,1]$ there exists $C>0$ and a deterministic $O(n^r)$-time algorithm whose input is an $n$-vertex $r$-uniform hypergraph $G$ and whose output is either `Fail' or a reservoir path $\Pres$ with ends $\tpl{u}$ and $\tpl{v}$ and an $(r-1)$-uniform exposure hypergraph $\cE$ on vertex set $V(G)$ with the following properties.
	\begin{enumerate}[label=\rom]
		\item\label{resp:inS} All vertices of $\Pres$ and edges of $\cE$ are contained in a set $S$ of size at most $\tfrac{n}{4}$.
		\item\label{resp:sizeR} The reservoir $R\subset V(\Pres)$ has size $2 C p^{-1} \log n$.
		\item\label{resp:ends} There are no edges of $\cE$ contained in $R\cup\tpl{u}\cup\tpl{v}$.
		\item\label{resp:exp} All $r$-sets in $V(G)$ which have been exposed contain at least one edge of $\cE$.
	\end{enumerate}
	
	When $G$ is drawn from the distribution $\cG^{(r)}(n,p)$ and $p\ge Cn^{-1}\log^3n$, the algorithm returns `Fail' with probability at most $n^{-2}$.
\end{lemma}

Furthermore we need a lemma which allows us to connect two given tuples with a not too long path.
This lemma is the engine behind the proof and behind the Reservoir Lemma.

\begin{lemma}[Connecting Lemma]\label{lem:connect}
For each $r\ge 3$ there exist $c,C>0$ and a deterministic $O(n^{r-1})$-time algorithm whose input is an $n$-vertex $r$-uniform hypergraph $G$, a pair of distinct $(r-1)$-tuples $\tpl{u}$ and $\tpl{v}$, a set $S\subset V(G)$ and an $(r-1)$-uniform exposure hypergraph $\cE$ on the same vertex set $V(G)$.
The output of the algorithm is either `Fail' or a tight path of length $o(\log n)$\footnote{We will make this more precise later. You could replace this by at most $C n / \log \log n$.} in $G$ whose ends are $\tpl{u}$ and $\tpl{v}$ and whose interior vertices are in $S$, and an exposure hypergraph $\cE'\supset\cE$.
We have that all the edges $E(\cE')\setminus E(\cE)$ are contained in $S\cup\tpl{u}\cup\tpl{v}$.

Suppose that $G$ is drawn from the distribution $\cG^{(r)}(n,p)$ with $p\ge C(\log n)^3/n$, that $\cE$ does not contain any edges intersecting both $S$ and $\tpl{u} \cup \tpl{v}$.
If furthermore $|S| = Cp^{-1}\log n$ and $|e(\cE[S])| \le c |S|^{r-1}$ then $e(\cE')\le e(\cE)+ O(|S|^{r-2})$ and the algorithm returns `Fail' with probability at most $n^{-5}$.
\end{lemma}

\subsection{Overview continued: more details}

We now describe the algorithm claimed by Theorem~\ref{thm:main}, which we state in a high-level overview as Algorithm~\ref{alg:tightHamCycle} and explain somewhat informally some of the arguments.

  \begin{algorithm}[h]
   \caption{Find a tight Hamilton cycle in $\cG^{(r)}(n,p)$}
    \label{alg:tightHamCycle}
      \lnl{step:reservoir} use subroutine from Lemma~\ref{lem:respath} to
    either construct $\Pres$ (with ends $\tpl{u}$, $\tpl{v}$ and exposure hypergraph $\cE$ on $S$) or halt with \Failure \;
     $L:=V(G)\setminus S$\; 
     $U:=S\setminus V(\Pres)$\;
     \lnl{step:coverU} extend $\Pres$ greedily from $v$ to cover all vertices of $U$ and using up to $n/2$ vertices of $L$, otherwise halt with \Failure \;
     \lnl{step:coverL} extend $\Pres$ further greedily to $\Palm$ by covering all vertices of $L$ and using up to $|R|/2$ vertices of $R$, otherwise halt with \Failure \;
     \lnl{step:connect_ends} use subroutine of Lemma~\ref{lem:connect} to connect the ends of $\Palm$ using the unused at least $|R|/2$ vertices of $R$, otherwise halt with \Failure \;
 \end{algorithm}

\subsubsection*{Step~1} Given $G$ drawn from the distribution $\cG^{(r)}(n,p)$, we begin by applying Lemma~\ref{lem:respath} to a.a.s.\ find a reservoir path $\Pres$ with ends $\tpl{u}$ and $\tpl{v}$ contained in a set $S$ of size $\tfrac{n}{4}$.
Let $L=V(G)\setminus S$, and $U=S\setminus V(\Pres)$.
Recall that by Lemma~\ref{lem:respath}~\ref{resp:inS} and~\ref{resp:ends}, all edges of $\cE$ are contained in $S$; and $R\cup\tpl{u}\cup\tpl{v}$ contains no edges of $\cE$.
By~\ref{resp:exp} all exposed $r$-sets contain an edge of $\cE$; by choosing a little carefully where to expose edges (see Step~2 below), we will not need to worry about what exactly the edges of $\cE$ are beyond the above information.

\subsubsection*{Step~2}We extend $\Pres:=P_0$ greedily, one vertex at a time, from its end $\tpl{u}=\tpl{u}_0$, to cover all of $U$.
At each step $i$, we simply expose the edges of $G$ which contain the end $\tpl{u}_{i-1}$ of $P_{i-1}$ and whose other vertex is not in $V(P_{i-1})$, choose one of these edges $e$ and add the vertex from $e\setminus \tpl{u}_{i-1}$ to $P_{i-1}$ to form $P_i$.
The rule we use for choosing $e$ is the following: if $i$ is congruent to $1$ or $2$ modulo $3$, we choose $e$ such that $e\setminus \tpl{u}_{i-1}$ is in $L$, and if $i$ is congruent to $0$ modulo $3$ we choose $e$ such that $e\setminus \tpl{u}_{i-1}$ is in $U$ if it is possible; if not we choose $e$ such that $x_i:=e\setminus \tpl{u}_{i-1}$ is in $L$.
The point of this rule is that at each step we want to choose an edge which contains at least two vertices of $L$, because no such $r$-set can contain an edge of $\cE$ since all the edges of $\cE$ are contained in $S$ (Property~\ref{resp:inS}).
We will see that while $U\setminus V(P_{i-1})$ is large, we always succeed in choosing a vertex in $U$ when $i$ is congruent to 0 modulo $3$.
When it becomes small we do not, but a.a.s.\ we succeed often enough to cover all of $U$ while using not more than $\tfrac{5n}{8}$ vertices of $L$.

\subsubsection*{Step~3}Next, we continue the greedy extension, this time choosing a vertex in $L$ when possible and in $R$ when not, until we cover all of $L$.
It follows from the first two steps and Properties~\ref{resp:inS} and~\ref{resp:ends} that no edge of $\cE$ is in $L\cup R$.
Thus, at each step we choose from newly exposed edges and again we a.a.s.\ succeed in covering $L$ using only a few vertices of $R$.
Let the final almost-path (which uses some vertices $R_1' \subseteq R$ twice) be $\Palm$, and $R_1$ the subset of $R$ consisting of vertices we did not use in the greedy extension, i.e.\ $R_1=R\setminus R_1'$.

\subsubsection*{Step~4} At last, $\Palm$ covers $V(G) = L \cup U \cup V(\Pres)$.
Its ends, together with the vertices of $R_1$, satisfy the conditions of Lemma~\ref{lem:connect}, which we apply to a.a.s.\ complete $\Palm$ to an almost-tight cycle  $H'$ in which some vertices of $R_1$ are used twice.
The reservoir property of $R$ now gives a tight Hamilton cycle $H$.

\subsubsection*{Runtime} Our applications of Lemmas~\ref{lem:respath} and~\ref{lem:connect} take time polynomial in $n$ by the statements of those lemmas; the greedy extension procedure is trivially possible in $O(n^2)$ time (since at each extension step we just need to look at the neighbourhood of an $(r-1)$-tuple, and there are $O(n)$ steps).
Finally the construction of $\Pres$ allows us to obtain $H$ from $H'$ in time $O(n^2)$: we scan through $\Pres$, for each vertex $r$ of $R$ we scan the remainder of $H'$ to see if it appears a second time, and if so locally reorder $V(\Pres)$ to remove $r$ from $\Pres$.

To prove Theorem~\ref{thm:main}, what remains is to justify our claims that various procedures above a.a.s.\ succeed.

\subsection{Proof of Theorem~\ref{thm:main}}
We choose $C\ge \max \{ C_{L_{\ref{lem:respath}}}, C_{L_{\ref{lem:connect}}}, 10^8 \}$ large enough for Lemmas~\ref{lem:respath} and~\ref{lem:connect} to hold.
For this proof we do not need to know the value of $c'$ required for Lemma~\ref{lem:connect}.
We suppose that $n$ is large enough to make $\log \log n$ larger than any constant appearing in the following proof.
 
 \subsubsection*{Constructing $\Pres$}
 Let $G$ be drawn from the distribution $\cG^{(r)}(n,p)$.
 Lemma~\ref{lem:respath} states that with probability at least $1-n^{-2}$, a reservoir path $\Pres$ in $G$ is found in polynomial time.
 From this point on, at each step except the final connection, when we expose edges at an $(r-1)$-set $\tpl{x}$, that $(r-1)$-set will be included in the path we construct.
 Hence in future steps we will not examine edges containing $\tpl{x}$.
 Thus while we should keep updating $\cE$, in fact we will never need to know which edges are added after generating $\Pres$.
 
\subsubsection*{Extending $\Pres$ to cover all of $U$} 
We next aim to prove that with high probability the greedy extension of $\Pres$ to cover $U$ succeeds, with at least $n/8$ vertices of $L$ remaining uncovered at the end.
Recall that we chose $|S|=\tfrac{n}{4}$ and thus $|L|=\tfrac{3n}{4}$.
We choose the next vertex from $L$ when $i$ is congruent to $1$ or $2$ modulo $3$ or when we fail to extend into $U$.
At each step $i$ where at least $n/8$ vertices of $L$ are uncovered, we expose all the $r$-sets in $V(G)$ which contain the end $u_{i-1}$ of $P_{i-1}$ and a vertex of $L$.
The greedy algorithm can only fail to complete step $i$ if none of these $r$-sets turn out to be edges, which happens with probability at most $(1-p)^{n/8}\le \exp\big(-\tfrac{pn}{8}\big)<n^{-4}$ (since the edges of the random hypergraph are independent).
Taking the union bound, the greedy algorithm to cover $U$ fails before covering $\tfrac58n$ vertices of $L$ with probability at most $n^{-3}$.
 
 Similarly, for any $i$ such that $\big|U\setminus V(P_{i-1})\big|\ge C p^{-1} \log n$, if $i$ is divisible by $3$ the probability that no edge containing $u_{i-1}$ and a vertex of $U\setminus V(P_{i-1})$ is in $G$ is at most $\exp\big(-C \log n\big)<n^{-4}$.
 It follows that with probability at most $n^{-3}$ the greedy algorithm chooses a vertex of $L$ when $i$ is divisible by $3$ and $U\setminus V(P_{i-1})$ has size at least $C p^{-1} \log n$.
 Let $t_1$ be the first time in the greedy extension procedure when $U\setminus V(P_{t_1})$ has size less than $C p^{-1} \log n$.
 
 It remains to show that while the last $C p^{-1} \log n$ vertices of $U$ are covered, at most $n/8$ vertices of $L$ are used.
 We split these last $C p^{-1} \log n$ vertices into the last $\tfrac12 p^{-1}$ vertices and the rest.
 When $x$ vertices of $U$ remain uncovered with $x \ge \tfrac12p^{-1}$, then the probability of choosing a vertex of $U$ for the vertex $x_i$ extending $P_{i-1}$ (when $i$ is divisible by $3$) is at least $1-(1-p)^x\ge \tfrac13$.
 By Chernoff's inequality, the probability that at time $t_2:=t_1+6 C p^{-1} \log n$ there are more than $\tfrac12p^{-1}$ vertices of $U$ remaining uncovered is at most $\exp\big(-\tfrac{1}{6}C p^{-1} \log n\big)\le n^{-3}$.
 Next, we show that we cover all but at most $\log n$ vertices of $U$ in not too much more time.
 
 To see this, consider the following event.
 For $1\le j\le 7n/8$ and $\log n\le x\le \tfrac12p^{-1}$, let $A(x,j)$ be the event that we have $\big|U\setminus V(P_j)\big|=x$ and $\big|U\setminus V(P_{j-3000p^{-1}})\big|\le 2x$.
 We claim that the probability for any of these events to hold is at most $n^{-3}$.
 Indeed, if for some given $x$ and $j$ the event $A(x,j)$ occurs, then at each of the at least $500p^{-1}$ values of $i$ with $j-3000p^{-1}\le i\le j$, an edge containing $\tpl{u}_{i-1}$ and a vertex of $U$ appears with probability at least $1-(1-p)^x\ge px/2$ (since $x\le\tfrac12p^{-1}$).
 Thus for $A(x,j)$ to hold, it is necessary that a sum of at least $500p^{-1}$ Bernoulli random variables, each with probability at least $px/2$, is at most $x$.
 Chernoff's inequality states that this probability is at most $\exp\big(-\tfrac{250x}{12}\big)\le n^{-5}$, and taking the union bound over all $A(x,j)$ the claim follows.
 Taking in particular $x=2^{-k}n/\log n$ for $k\ge 1$ such that $2^{-k}n\log n\ge\log n$ (so $k\le\log n$) we see that with probability at least $1-n^{-3}$, at time $t_3:=t_2+3000p^{-1}\log n$ there are at most $\log n$ vertices of $U$ remaining uncovered.

While at least one vertex of $U$ remains uncovered, the probability that when $i$ is divisible by three we choose a vertex of $U$ is at least $p$.
Applying Chernoff's inequality, the probability that at time $t_4:=t_3+300p^{-1}\log n$ we still have not covered all of $U$ is at most $\exp(-\tfrac{100\log n}{12} ) \le n^{-3}$.
Putting all this together, the probability that $V(P_{t_4})$ does not cover $U$ is at most $4n^{-3}$.
Since $t_1\le 3|U|$, since $|U|\le|S|\le n/4$, and since $t_4-t_1\le n/16$, we conclude that with probability at least $1-4n^{-3}$ the greedy extension procedure indeed covers $U$ with at least $n/8$ vertices of $L$ left uncovered.
Let $t_5$ be the first time at which $P_{t_5}$ covers $U$.

\subsubsection*{Extending $\Pres$ further to $\Palm$ by covering all of $L$}
We now repeat a similar procedure to use up all of $L\setminus V(P_{t_5})$ while not using too many vertices in $R$.
Since no edges of $\cE$ are contained in $R\cup L$, at each time $t$, all the $r$-sets containing the end $\tpl{u}_{t-1}$ of $P_{t-1}$ and a vertex of $L\cup R\setminus V(P_{t-1})$ are unrevealed.
In particular, provided that at each step we have $\big|R\setminus V(P_{t-1})\big|\ge \tfrac12|R|$, by Chernoff's inequality with probability at least $1-n^{-4}$ at least one edge of $G$ is found consisting of $\tpl{u}_{t-1}$ and a vertex of $R\setminus V(P_{t-1})$.
Taking the union bound, the probability of the extension procedure failing when $\big|R\setminus V(P_{t-1})\big|\ge \tfrac12|R|$ is at most $n^{-3}$.

As long as $\big|L\setminus V(P_{t-1})\big|\ge \tfrac{C}{100} p^{-1} \log n$, by Chernoff's inequality with probability at most $\exp\big(-\tfrac{C}{300} \log n\big)\le n^{-4}$ there is no edge of $G$ containing $\tpl{u}_{t-1}$ and a vertex of $L\setminus V(P_{t-1})$; in particular with probability at least $1-n^{-3}$ the greedy extension covers all but at most $\tfrac{C}{100} p^{-1} \log n$ vertices of $L$ before using any vertex of $R$.
Let $t_6$ be the time at which all but at most $\tfrac{C}{100} p^{-1} \log n$ vertices of $L$ are covered.
Again, we now consider the time taken to cover all but $\tfrac12p^{-1}$ vertices of $L$.
At each time the probability of being able to choose a vertex of $L$ to extend our path with is at least $\tfrac13$, so that with probability at least $1-n^{3}$ we cover all but at most $\tfrac12p^{-1}$ vertices of $L$ by time $t_7\le t_6+\tfrac{C}{25} p^{-1} \log n$.
In particular we use at most $\tfrac{C}{25} p^{-1} \log n$ vertices of $R$ in this time.
  
  By the same analysis as before, the total time taken to go from covering all but at most $\tfrac12p^{-1}$ vertices of $L$ to covering all but at most $\log n$ vertices of $L$ and then all vertices of $L$ is with probability at least $1-2n^{-3}$ not more than $3000p^{-1}\log n+300p^{-1}\log n$.
  Putting this together, provided all these good events hold we succeed in covering all but at most $\log n$ vertices of $L$ having used at most
  \[\tfrac{C}{25} p^{-1} \log n+3300p^{-1}\log n < C p^{-1} \log n = \tfrac12|R| \]
  vertices of $R$.
  
  In sum, with probability at least $1-n^{-2}-8n^{-3}$, the algorithm succeeds in generating $\Palm$, where the set $R'\subset R$ of vertices not used in the greedy extension has size at least $\tfrac12|R|$.
  
\subsubsection*{Connecting the end tuples of $\Palm$ and getting the tight Hamilton cycle} 
  Applying Lemma~\ref{lem:connect} to connect the end tuples of $\Palm$ in a subset of $R'$ of size $Cp^{-1}\log n$ (which is possible since $R'$ together with the ends of $\Palm$ contains no edges of $\cE$ and since $|R'|\ge n/\log^2 n$), with probability at least $1-n^{-4}$ we find the desired almost-tight cycle $H'$, which gives us deterministically the desired tight Hamilton cycle $H$.
  Thus as desired the probability that our algorithm fails to find a tight Hamilton cycle is at most $n^{-1}$.
  \qed

\section{Proof of the Connecting Lemma}

In this section we prove Lemma~\ref{lem:connect} and a very similar lemma (Lemma~\ref{lem:spike}) dealing with `spike-paths' which we will require for Lemma~\ref{lem:respath}.
A spike-path is similar to a tight path, but after $(r-1)$-steps the direction of the last $(r-1)$-tuple is inverted.

\begin{definition}[Spike path]
	In an $r$-uniform hypergraph, a spike path of length $t$ consists of a sequence of $t$ pairwise disjoint $(r-1)$-tuples $\tpl{a}_1,\dots,\tpl{a}_t$, where $\tpl{a}_i=(a_{i,1},\dots,a_{i,r-1})$ for all $i$, with the property, that the edges $\{ a_{i,r-j},\dots,a_{i,1},a_{i+1,1},\dots,a_{i+1,j} \}$ are present for all $i=1,\dots,t-1$ and $j=1,\dots,r-1$.
	We call $\tpl{a}_i$ the $i$-th spike.
\end{definition}

This is the same as taking $t$ tight paths of length $2(r-1)$, where the end $(r-1)$-tuples of path $i$ are $\tpl{x}_i$ and $\tpl{y}_i$, and identifying $\rev{\tpl{x}_i}$ with $\tpl{y}_{i+1}$ for all $i=1,\dots,t-1$.
The proofs of Lemmas~\ref{lem:connect} and~\ref{lem:spike} are essentially identical, so we give the details of the former and then explain how to modify it to obtain the latter.

\subsection{Preliminaries}

For an $(r-1)$-tuple $\tpl{u}$ and an integer $i$ we define a \emph{fan} $\cF_i(\tpl{u})$ in an $r$-uniform hypergraph $\cH$ as a set $\{ P_1, \dots , P_s \}$ of tight paths in $\cH$, of length $i$ or $i+1$, starting in $\tpl{u}$.
For any set or tuple $\tpl{a}$, let $\{ P_j \}_{j \in I}$ be the subcollection of tight paths from $\cF_i(\tpl{u})$ in which $\tpl{a}$ appears as a consecutive interval (in arbitrary order).
The \emph{leaves} or \emph{ends} of $\cF_i(\tpl{u})$ are the ending $(r-1)$-tuples of alle the paths $P_1, \dots, P_s$.
We denote by $\mult(\tpl{a})$ the number of different
paths we see in $\{P_j\}_{j\in I}$ after truncating behind $\tpl{a}$.

In any $r$-uniform hypergraph $H=(V,E)$ the degree of a set or tuple $f$ of size $1 \le |f| \le r-1$ is the number of edges which it is contained in, i.e.
\begin{align*}
\deg_H(f) = |\{ e \in E : f \subseteq e \}|.
\end{align*}
Given a set $S \subseteq V$, we write $\deg_H (f,S)$ for the degree into $S$, that is, where we count only edges $e$ satisfying $e\setminus f\subset S$.

\subsection{Idea and further notation}

The basic idea is that, starting with the $\tpl{u}$ and $\tpl{v}$ and the empty fans $\cF_0(\tpl{u})$ and $\cF_0(\tpl{v})$, we want to fan out.
That is, for each path in $\cF_i(\tpl{u})$ we will find a large collection of ways to extend by one vertex and all the resulting paths form $\cF_{i+1}(\tpl{u})$.
We do this until we have fans $\cF_t(\tpl{u})$ and $\cF_{t}(\tpl{v})$ with 
\begin{align*}
	Q:=p^{-(r-1)/2} \log n
\end{align*}
leaves each.
This happens roughly when we have 
\begin{align*}
t:= 2 \cdot \left\lceil\frac{\log (Q)}{\log(\log n)}\right\rceil \le (r-1) \cdot \left\lceil\frac{\log (p^{-1})}{\log(\log n)}\right\rceil + 2 = o(\log n).
\end{align*}

A complication is that in this process we have to avoid the edges of $\cE$ when expanding the fans.
In order to make the modifications for the promised spike-path variation easy (cf.\ Lemma~\ref{lem:spike} below), we will do something a little more complicated.
We split into expansion and continuation phases, each of length $r-1$.
The first phase is an expansion phase, so when forming $\cF_1(\tpl{u})$, \dots, $\cF_{r-1}(\tpl{u})$ we find many ways to extend each path by one vertex and put all of them into the next fan.
The second phase is a continuation phase, so when forming $\cF_r(\tpl{u})$, \dots, $\cF_{2r-2}(\tpl{u})$ we choose only one way to extend each path.
As soon as we have a collection of paths with the desired $Q$ leaves, we cease expanding (even if we are still in an expansion phase) and simply continue each path such that each has the same length.
We construct fans from $\tpl{v}$ similarly, and we continue construction up to $\cF_{t}(\tpl{v})$.

In the final step we find $r-1$ further edges connecting two of the leaves, giving us a tight path connecting $\tpl{u}$ to $\tpl{v}$.
Again there is a complication here: some pairs of leaves $(\tpl{w},\tpl{x})$ may be \emph{blocked} by edges of $\cE$, meaning that inside some $r$ consecutive vertices of the concatenation $\tpl{w}\rev{\tpl{x}}$ there is an edge of $\cE$.
If a pair of leaves is blocked, then trying to reveal $(r-1)$ edges connecting the pair would mean revealing an edge of the random hypergraph twice (and if a pair is not blocked then doing so does not reveal any edge twice).
We need to take this into account in our analysis, and we need to construct $\cF_{t}(\tpl{v})$ carefully to avoid creating \emph{dangerous} leaves for which a large fraction of the pairs is blocked.

To make this precise, we use the following algorithm.

  \begin{algorithm}[h]
   \caption{Find a connecting path from $\tpl{u}$ to $\tpl{v}$}
    \label{alg:connect}
     split $S$ into equal parts $S_1,\dots,S_{4(r-1)},S'_1,\dots,S'_{4(r-1)}$\;
     $\cF_t(\tpl{u}):=\mathrm{BuildFan}(\tpl{u},S_1,\dots,S_{4(r-1)},\emptyset)$\;
     set $D:=\big\{\tpl{x}\in S^{r-1}:(\tpl{w},\tpl{x})\text{ is blocked for at least } \xi'Q \text{ leaves $\tpl{w}$ of $\cF_t(\tpl{u})$}\big\}$\;
     $\cF_{t}(\tpl{v}):=\mathrm{BuildFan}(\tpl{v},S'_1,\dots,S'_{4(r-1)},D)$\;
     find $r-1$ edges connecting a leaf of $\cF_t(\tpl{u})$ to the reverse of one of $\cF_{t}(\tpl{v})$\;
     \Return tight path $P$ connecting $\tpl{u}$ to $\tpl{v}$ \;
 \end{algorithm}
 
 The subroutine $\mathrm{BuildFan}$ takes as input a starting tuple, the sets in which to build a fan, and a \emph{danger hypergraph} $D$ which is important for the construction of the second fan: it is an $(r-1)$-uniform hypergraph which records the tuples in $S'_1,\dots,S'_{4(r-1)}$ which we cannot easily connect to the leaves of $\cF_t(\tpl{u})$.
 The algorithm ensures that no leaf of a fan will be a dangerous tuple.
 Though we only need this for the leaves of the final fan, it is convenient to maintain this property throughout.
 For convenience, we write $S_i$ for the set $S_{i\, \mathrm{mod}\, 4(r-1)}\in\{S_1,\dots,S_{4(r-1)}\}$ with $S_0=S_{4(r-1)}$; the point of these sets is that we choose the $i$th vertex of each path in $S_i$, which is helpful in the analysis.
 Finally, we need to ensure that we always choose `good' vertices which allow us to continue our construction and prove various probabilistic statements.
 To that end, we define a vertex $b$ to be \emph{good} with respect to an exposure hypergraph $\cE$, a set $\cF$ of paths with distinct ends, a danger hypergraph $D$ and a $(r-1)$-tuple $\tpl{a}$ if none of the following statements hold for any (possibly empty) tuple $\tpl{c}$ whose vertices are contained in those of $\tpl{a}$ (not necessarily in the same order).
\begin{enumerate}[label=\rom]
	\item
	$b$ appears somewhere on the unique path $P(\tpl{a})$ ending in $\tpl{a}$,
	\label{bad:path}
	\item  
	$|\tpl{c}|\le r-2$ and $\deg_\cE( \{\tpl{c}, b \},S) > \xi^{r-|\tpl{c}|-1} |S|^{r-|\tpl{c}|-2}$,
	\label{bad:degree}
	\item $\mult( \{ \tpl{c}, b \} )  > \xi^{r-|\tpl{c}|-1} Q \cdot |S|^{-|\tpl{c}|-1} \cdot \log^{|\tpl{c}|+1} n$, and
	\label{bad:mult}
	\item $|\tpl{c}|\le r-2$ and $\deg_D( \{\tpl{c}, b \},S) > (\xi'|S|)^{r-|\tpl{c}|-2}$.
	\label{bad:dangerous}
\end{enumerate}
 Normally $\cE$, $\cF$ and $D$ will be clear from the context and we will simply say good for $\tpl{a}$.
 We are finally ready to give the $\mathrm{BuildFan}$ subroutine.
 
 \begin{algorithm}[h]
  \caption{$\mathrm{BuildFan}(\tpl{s},T_1,\dots,T_{4(r-1)},D)$}
  \label{alg:build}
  $\cF_0:=\big\{\tpl{s}\big\}$\;
  \ForEach{$i=1,\dots,t$}{
   \uIf{$i\mod 2(r-1)\in\{1,\dots,r-1\}$}{
    $\mathrm{phase}=$`expand'\;
   }
   \Else{
    $\mathrm{phase}=$`continue'\;
   }
   $\mathrm{NumPaths}:=|\cF_{i-1}|$\;
   $\cF_i:=\cF_{i-1}$\;
   \ForEach{$P\in\cF_{i-1}$}{
    \lnl{line:a} let the $(r-1)$-tuple $\tpl{a}$ be the end of $P$ \;
    reveal the edges of $G$ containing $\tpl{a}$ and add $\tpl{a}$ to $\cE$ \;
    \lnl{line:b} let $T \subseteq T_i$ be the set of vertices $b$ which are good for $\tpl{a}$ and $\{ \tpl{a},b \}$ is an edge\; 
    \uIf{$\mathrm{phase}=$`expand'}{
     $\mathrm{Add}:=\min\big(\log n,Q+1-\mathrm{NumPaths}\big)$\;
     choose $\mathrm{Add}$ vertices $b_1,\dots,b_{\mathrm{Add}}\in T$\;
     $\cF_i:=\cF_i\cup\{(P,b_1),\dots,(P,b_{\mathrm{Add}})\}\setminus\{P\}$\;
     $\mathrm{NumPaths}:=\mathrm{NumPaths}+\mathrm{Add}-1$\;
    }
    \Else{
     choose a vertex $b\in T$\;
     $\cF_i:=\cF_i\cup\{(P,b)\}\setminus\{P\}$\;
    }
   }
  }
  \Return $\cF_t$ \;
 \end{algorithm}

\subsection{Proof}
We set
\begin{align}
\label{eq:constants}
\xi' = \tfrac{1}{100^r}\,, \quad \xi=(\xi')^r/(2r2^{20r})\,, \quad \delta=8^r \xi+\xi'\,,\quad  C=10^{8r}\,\quad\text{and} \quad c = 10^{-r} \xi^{r}\,.
\end{align}

The proof amounts to showing two things.
First, $\mathrm{BuildFan}$ is likely to succeed---that is, that it does not fail for lack of good vertices before returning a fan, that the returned fan does have size $Q$, and that it does not add too many tuples to $\cE$.
Second, the required extra $r-1$ edges which should connect the fans can be found.

\subsubsection*{Creating the fans}
We begin by showing that the subroutine $\mathrm{BuildFan}(\tpl{s},T_1,\dots,T_{4(r-1)},D)$ is likely to succeed, whether we choose $\tpl{s}=\tpl{u}$, $T_i=S_i$ and $D=\emptyset$ or we choose $\tpl{s}=\tpl{v}$, $T_i=S'_i$ and $D$ as given in Algorithm~\ref{alg:connect}, using the following claim.

We define $L_i$ to be the leaves of $\cF_i$.

\begin{claim}
	\label{claim:fan}
	If step $i$ was successful, then step $i+1$ is successful with probability at least $1-n^{-3r}$ and the following holds throughout step $i+1$ for each $\tpl{a} \in L_{i+1}$ and each non-empty $\tpl{c}$ whose vertices are chosen from $\tpl{a}$, not necessarily in the same order.
	\begin{enumerate}[label = \bfseries P\arabic{enumi}]
		\item Each path in $\cF_i$ extends to at least one path in $\cF_{i+1}$; if $2(r-1)\ell<i\le 2(r-1)\ell+r-1$ and $|\cF_{i+1}|<Q$ then each path in $\cF_i$ extends to at least $\log n$ paths in $\cF_{i+1}$.
		In both cases, all leaves are not in $\cE$.
		\label{claim:fan_expansion}
		\item $e(\cE[S]) \le c |S|^{r-1} + 20r Q$.
		\label{claim:fan_edges}
		\item If $|\tpl{c}|<r-1$ we have $\deg_\cE (\tpl{c},S) \le \xi^{r-|\tpl{c}|} |S|^{r-1-|\tpl{c}|} +1$.
		\label{claim:fan_extension}
		\item We have $\mult(\tpl{c}) \le \xi^{r-|\tpl{c}|} Q \cdot |S|^{-|\tpl{c}|} \cdot \log^{|\tpl{c}|} n + 1$.
		\label{claim:fan_collapsing}
		\item If $1\le|\tpl{c}|\le r-2$ we have $\deg_D( \tpl{c} ,S) \le (\xi'|S|)^{r-|\tpl{c}|-1}$.
	  \label{claim:fan_dangerous}
	\end{enumerate}
\end{claim}
\begin{claimproof}[Proof of Claim~\ref{claim:fan}]
  Observe that $\cF_0$ trivially satisfies the conditions of Claim~\ref{claim:fan}, modulo Chernoff's inequality for~\ref{claim:fan_expansion}.
  Suppose that for some $0\le i<t$, at each step $0\le j\le i$ of Algorithm~\ref{alg:build} the conditions of Claim~\ref{claim:fan} are satisfied.
  In particular, by~\ref{claim:fan_collapsing}, the ends of the paths $\cF_i$ are distinct as for $|\tpl{c}|=r-1$ we have $\mult({\tpl{c}}) <2$, and by~\ref{claim:fan_expansion} we have $|\cF_i|\ge\min\big(\log^{i/2} n,Q\big)$.
  
  To begin with, we show that $\cE$ cannot have too many edges.
  At each step $j$ with $1\le j\le i$, we add $|\cF_{j-1}|$ edges to $\cE$, so that we want to upper bound $\sum_{j=1}^t|\cF_{j-1}|$.
  Definitely $\cF_t$ has size at most $Q$ and $\cF_{j-4(r-1)}$ always has size less than half of $\cF_j$, so that this sum is dominated by $4r\sum_{i=1}^\ell 2^i$ where $\ell=\log_2Q$.
  We conclude that $\sum_{j=1}^t|\cF_{j-1}|\le 8rQ$.
  Since we create two fans, in total we obtain the claimed bound~\ref{claim:fan_edges}.
  
 \medskip 
  
  We now show that, for each choice of $P\in\cF_i$ with end $\tpl{a}$, the total number of vertices in $T_{i+1}$ which are not good for $\tpl{a}$ is at most $\delta|S|$.
  This will allow us to prove~\ref{claim:fan_expansion}.
  First, since $P$ has at most $t$ vertices, at most $t$ vertices are excluded by~\ref{bad:path}.
  
	For each $\tpl{c}$ of size at most $r-2$ with vertices chosen from $\tpl{a}$, there are at most $2r \xi |S|$ vertices fulfilling~\ref{bad:degree}.
	To see this for $|\tpl{c}|=0$, observe that otherwise we have $e(\cE[S]) > 2 \xi^{r} |S|^{r-1} > 2 c|S|^{r-1}$, contradicting~\ref{claim:fan_edges} as $Q \le \tfrac1C |S|^{r-1}$.
	Assume that it fails for some non-empty $\tpl{c}$.
	Then there are more than $2r\xi |S|$ vertices $x \in T_{i+1}$ with 
	\begin{align*}
	\deg_\cE( \{\tpl{c}, x \},S) > \xi^{r-|\tpl{c}|-1} |S|^{r-|\tpl{c}|-2}
	\end{align*}
	which implies that 
	\begin{align*}
	\deg_\cE( \tpl{c},S) > 2 \xi^{r-|\tpl{c}|} |S|^{r-|\tpl{c}|-1}
	\end{align*}
	in contradiction to~\ref{claim:fan_extension}.
	
	Furthermore there are at most $2r\xi |S|$ vertices $b$ fulfilling~\ref{bad:mult} for each $\tpl{c}$.
	Again for $|\tpl{c}|=0$ it is enough to note that there are at most $Q$ paths in total and thus there are at most
	\begin{align*}
	\frac{Q}{\xi^{r-1}Q \cdot |S|^{-1} \cdot \log n} \le \xi |S|
	\end{align*}
	vertices $b$ with $\mult(b)  > \xi^{r-1}Q \cdot |S|^{-1} \cdot \log n$.
	Now suppose $\tpl{c}$ is not empty.
	Every path in $\cF_{i+1}$ whose end contains $\{ \tpl{c},b \}$ was constructed by the expansion of some path in $\cF_i$ whose end contains $\tpl{c}$.
	Note that every path expands at most by a factor of $\log n$ and by~\ref{claim:fan_extension} there are at most $\xi^{r-|\tpl{c}|}Q \cdot |S|^{-|\tpl{c}|}  \log^{|\tpl{c}|}n+1$ paths in $\cF_i$ whose end contains $\tpl{c}$.
	If this bound is less than two, then there are at most $\log n$ vertices $b$ with $\mult( \{ \tpl{c},b \} ) \ge 1$.
	Otherwise there are at most
	\begin{align*}
	\frac{2 \xi^{r-|\tpl{c}|} Q \cdot |S|^{-|\tpl{c}|} \log^{|\tpl{c}|+1}n}{\xi^{r-|\tpl{c}|-1} Q \cdot |S|^{-|\tpl{c}|-1} \log^{|\tpl{c}|+1}n}= 2 \xi |S|
	\end{align*}
	vertices $x \in S_i$ with $\mult( \{ \tpl{c},b \} )  > \xi^{r-|\tpl{c}|-1} Q \cdot |S|^{-|\tpl{c}|-1} \cdot \log^{|\tpl{c}|+1} n$.
	
  Finally, we want to show that for each $\tpl{c}$ there are at most $\xi'|S|$ vertices $b$ in $T_i$ which satisfy~\ref{bad:dangerous}.
  This is trivial for $D=\emptyset$, so we may assume that $D$ is as given in Algorithm~\ref{alg:connect}.
  
First suppose $|\tpl{c}|=0$.
If a vertex $b$ satisfies~\ref{bad:dangerous}, then it is in $(\xi'|S|)^{r-2}$ edges of $D$, so if there are $\xi'|S|$ such vertices then there are at least $(\xi'|S|)^{r-1}$ edges in $D$ using vertices of $T_i$ (note that edges of $D$ only intersect $T_i$ in one vertex).
In other words, the number of blocked pairs $(\tpl{a},\tpl{b})$ with $\tpl{a} \in \cF_t(\tpl{u})$ and $\tpl{b} \in S^{r-1}$ is at least
\[(\xi'|S|)^{r-1}\cdot \xi'Q \ge 2r\cdot 2^{2r}\xi |S|^{(r-1)} \cdot Q\]
using our choice of parameters \eqref{eq:constants}.
We conclude that there is a leaf $\tpl{a}$ of $\cF_t(\tpl{u})$ that is in at least $2r\cdot 2^{2r}\xi|S|^{r-1}$ blocked pairs with tuples $\tpl{b}\in S^{r-1}$.
Fix this leaf.
Now~\ref{claim:fan_extension} holds for $\tpl{a}$, and we will show that this gives a contradiction.
Consider the following property of tuples $\tpl{b}$.
For any sets $A$ and $B$ with vertices in $\tpl{a}$ and $\tpl{b}$ respectively, if $|A|+|B|=r-1$ then $A\cup B$ is not in $\cE$, while if $|A|+|B| <  r-1$ then we have $\deg_\cE(A\cup B,S) \le 2\xi^{r-|A|-|B|} |S|^{r-1-|A|-|B|}$.
Trivially if $\tpl{b}$ has the property, then $(\tpl{a},\tpl{b})$ is not blocked.
If $\tpl{b}$ does not have the property, then let $B_{\tpl{b}}$ be a set of minimal size witnessing the property's failure.
Since $A\not\in\cE$ by~\ref{claim:fan_expansion}, and by~\ref{claim:fan_extension}, we do not have $|B_{\tpl{b}}|=0$.

We now count the ways to create $\tpl{b}$ which does not have the property.
We choose vertices $b_1,\dots,b_{r-1}$ one at a time until we create a witness $B\neq\emptyset$ that $\tpl{b}$ cannot have the property.
When we come to choose $b_j$, we have at most $|S|$ ways to choose it without creating a witness.
If we are to choose $b_j$ which witnesses the property's failure, then there are sets $A$ and $B'$ contained respectively in $\tpl{a}$ and $\{b_1,\dots,b_{j-1}\}$ such that $(A,B'\cup\{b_j\})$ fails the property.
There are at most $2^{2r}$ choices for $A$ and $B'$.
Since $(A,B')$ does not witness the property failing, by definition for each choice of $A$ and $B'$ there are at most $\xi|S|$ choices of $b_j$.
Summing up, there are at most $r\cdot 2^{2r}\xi|S|^{r-1}$ tuples $\tpl{b}$ which do not have the property.
As all blocked pairs use a tuple from this set, this is the desired contradiction.

Now suppose $\tpl{c}$ is a tuple for which there are at least $\xi'|S|$ vertices $b$ satisfying~\ref{bad:dangerous}.
In other words, there are more than $\xi'|S|$ vertices $b\in T_{i+1}$ with $\deg_D( \{\tpl{c}, b \},S) > (\xi'|S|)^{r-|\tpl{c}|-2}$, which implies that
\begin{align*}
\deg_D( \tpl{c},S) > (\xi'|S|)^{r-|\tpl{c}|-1}
\end{align*}
in contradiction to~\ref{claim:fan_dangerous}.

Putting all this together we conclude that there are at most $\delta|S|$ vertices $b$ such that $\tpl{c}$ exists satisfying any one of the conditions~\ref{bad:path}--\ref{bad:dangerous}, as desired.
	
	 \medskip 
	 
	Now let $\tpl{a}$ be a leaf of $\cF_i$.
	We now reveal all $r$-sets containing $\tpl{a}$ which were not revealed before and which use a vertex $x$ of $T_{i+1}$ which is good for $\tpl{a}$.
	Let $X$ be the number of edges $\{ \tpl{a},x \}$ which appear.
	Then the expected value of $X$ is at least $p (1-\delta) |T_{i+1}| \ge \tfrac{C}{20r} \log n$.
	Applying the Chernoff bound \eqref{eq:Chernoff} we get that $X < \tfrac{C}{40r} \log n$ with probability at most $2 \exp( - C \log n / (240r)) \le n^{-4r}$.
	Let us suppose that $X\ge\log n$.
	Then Algorithm~\ref{alg:build} does not fail to create the required number of paths from $\tpl{a}$.
	Taking a union bound over the at most $|S|^{r-1} t$ such events, we obtain the stated success probability of Claim~\ref{claim:fan}.
	
	It remains to prove that~\ref{claim:fan_extension},~\ref{claim:fan_collapsing} and~\ref{claim:fan_dangerous} also hold in $\cF_{i+1}(\tpl{u})$.
	But this is immediate, since we avoided choosing vertices which could cause their failure.
\end{claimproof}

Taking a union bound over the $2t$ steps, we conclude that with probability at most $n^{-2r}$ there is a failure to construct either of the desired fans $\cF_t(\tpl{u})$ and $\cF_{t}(\tpl{v})$.

\subsubsection*{Connecting the fans}
By construction, as set up in line~\ref{line:b} of Algorithm~\ref{alg:build}, all leaves of $\cF_{t}(\tpl{v})$ are not edges of $D$ and thus not dangerous.
Let $L$ be the leaves from $\cF_t(\tpl{u})$ and $L'$ the leaves from $\cF_{t}(\tpl{v})$ reversed.
We now want to reveal more edges to connect a leaf from $L$ with one from $L'$.

For $\tpl{a} \in L$ and $\tpl{b} \in L'$ let $P$ be the tight path with $r-1$ edges on the vertices $(\tpl{a},\tpl{b})$.
There are $|L'| \cdot (1 -\xi') |L| = (1 -\xi')Q^2$ many such paths $P$, which are not blocked, because $\tpl{b}$ is not dangerous.
Let $\cP$ be the set of all these paths which are not blocked.

Let $I_P$ be the indicator random variable for the event that the path $P$ appears, which occurs with probability $p^{r-1}$.
Further let $X$ be the random variable counting the number of paths which we obtain and note $X=\sum_{P \in \cP} I_P$.
With Janson's inequality \eqref{eq:Janson} we want to bound the probability that $X=0$.
First let us estimate the expected value of $X$.
By the observation from above we have $\EE (X) = |\cP| p^{r-1} \ge (1 -\xi')(cC)^{r-1} \log^{r-1}n \ge \log n$.

Now consider two distinct paths $P=(\tpl{a},\tpl{b})$ and $P'=(\tpl{a}',\tpl{b}')$, which share at least one edge.
It follows from property~\ref{claim:fan_collapsing} of Claim~\ref{claim:fan} and the quantities $Q$ and $|S|$, that two paths are identical if they share at least $r/2$ vertices in their end tuple.
Since either the start or end $r/2$-tuple of one of the $(r-1)$-tuples from $P$ has to agree with $P'$, we can assume without loss of generality that $\tpl{a}=\tpl{a}'$.
Further we can assume that for some $1 \le j < r/2$, $\tpl{b}$ and $\tpl{b}'$ agree on the first $j$ entries, but not in the $(j+1)$-st.
They can not share another $r/2$ or more entries as this would imply $\tpl{b}=\tpl{b'}$.
Thus $P$ and $P'$ share precisely an interval of length $r-1+j$ and thus $j$ edges.
With this we can bound $\EE (I_P I_{P'}) \le p^{2r-2-j}$.

Let $N_{P,j}$ be the number of paths $P'$ such that $P$ and $P'$ share precisely $j$ edges.
The above shows that for fixed $P=(\tpl{a},\tpl{b})$, $N_{P,j}$ is at most the number of choices of leaves $\tpl{b}' \in L'$ such that $\tpl{b}$ and $\tpl{b}'$ only differ in the ending $(r-1-j)$-tuple, plus the number of choices of leaves $\tpl{a}' \in L$ such that $\tpl{a}$ and $\tpl{a}'$ only differ in the start $(r-1-j)$-tuple.
It follows from property~\ref{claim:fan_collapsing} of Claim~\ref{claim:fan}, that the start $j$-tuple of $\tpl{b}'$ and the end $j$-tuple of $\tpl{a}'$ are the ends of at most $\xi^{r-j}Q\cdot |S|^{-j} \log^j n +1$ many paths.
This implies that $N_{P,j} \le Q\cdot |S|^{-j} \log^j n$, because $j< r/2$.

We can now obtain for $P,P'\in\cP$
\begin{align*}
\Delta=\sum_{P \not=  P',P\cap P'\not=\emptyset}\Exp(I_PI_{P'})
=\sum_{P \in \cP}\sum_{1\le j<r/2}\Big(\sum_{|P'\cap P|=j} \Exp(I_PI_{P'})\Big).
\end{align*}
With the above we get
\begin{equation*}
\begin{split}
\Delta&\le\sum_{P \in \cP}\sum_{1\le j<r/2} N_{P,j} \cdot p^{2r-2-j}\\
&\le |\cP|^2 p^{2r-2} \sum_{1\le j<r/2}|\cP|^{-1}  \cdot Q\cdot |S|^{-j} \log^j n \cdot p^{-j} \\
& \le \EE (X)^2  \cdot 2 \, Q^{-1}  \sum_{1\le j<r/2} C^{-j} \le \EE (X)^2 3\, C^{-1} \log^{-1}n,
\end{split}
\end{equation*}
where we used that $|S| \ge C p^{-1} \log n$ and $Q \ge \log n$.
Hence, Janson's inequality~\eqref{eq:Janson} implies that $\Prob(X=0)\le \exp(-\EE (X)^2/(\EE (X) + \Delta))\le \exp(-\tfrac{C}{6} \log n)$.
Thus we find some connection with probability at least $1-n^{-2r}$.

But we do not want to reveal all the $O(Q^2)$ edges for all paths from $\cP$, since this would add way to manu edges to the exposure hypergraph $\cE$.
The above argument proves that it is very likely that the desired connecting path exists and we will argue how to find such a path in an ``economic'' way.
We find it by the following procedure.
First we reveal all the edges at each leaf in $L$ and $L'$.
This entails adding $2 Q$ edges to $\cE$ and if $r=3$ then we are already done and we have added $2Q \le |S|$ edges to $\cE$.

For $r \ge 4$ we then construct from each leaf of $L$ all possible tight paths in $S$ with $\lfloor (r-2)/2 \rfloor$ edges and similarly from each leaf of $L'$ all tight paths of length $\lfloor (r-3)/2 \rfloor$.
We do this by the obvious breadth-first-search procedure, revealing at each step all edges at the end of each currently constructed path with less than $\lfloor (r-2)/2 \rfloor$ (or $\lfloor (r-3)/2 \rfloor$ respectively) edges which have not so far been revealed and adding each end to $\cE$.
Trivially, if the desired path exists then two of these constructed paths will link up, so that this procedure succeeds in finding a connecting path with probability $1-n^{-2r}$.

The expected number of edges in $S$ containing any given $(r-1)$-set in $S$ is $p(|S|-r+1)$, is between $\tfrac{C}{2}\log n$ and $C\log n$.
Thus by Chernoff's inequality and the union bound, with probability at least $1-n^{-3r}$ no such $(r-1)$-set is in more than $2C\log n$ edges contained in $S$.
It follows that the number of edges we add to $\cE$ in this procedure is with probability at least $1-n^{-3r}$ not more than
\begin{equation*}
\begin{split}
2 Q\sum_{i=0}^{\lfloor (r-2)/2 \rfloor }(2C\log n)^i &\le 2 p^{-(r-1)/2} \log n \cdot r (2C \log n)^{(r-2)/2} \\
&= O\left(  p^{-(r-2)} \log^{r-2} n  \right) = O(|S|^{r-2})\,,
\end{split}
\end{equation*}
for $r \ge 4$.
Putting this together with property~\ref{claim:fan_edges} of Claim~\ref{claim:fan} we see that the final exposure graph $\cE'$ has at most $O(|S|^{r-2})$ edges more than $\cE$, as desired.

\subsubsection*{Probability and runtime}
Altogether we have that our algorithm for the Connecting Lemma fails with probability at most $n^{-2r}+n^{-2r}+n^{-3r } \le n^{-5}$.

We now estimate the running time of our algorithm.
In total we added $O(|S|^{r-2})$ many $(r-1)$-tuples to $\cE$.
For every $(r-1)$-tuple exposed, we have to go through at most $n$ vertices until we found all new edges.
This gives at most $O(n^{r-1})$ steps.
We can easily keep track of the bounds for Claim~\ref{claim:fan} and update them after each event.
Since there is nothing else to take care of, we have a total number of at most $O(n^{r-1})$ steps.

\subsection{Spike path version}

The statement of the lemma is almost the same as for the tight path version, Lemma~\ref{lem:connect}.

\begin{lemma}[Spike path Lemma]\label{lem:spike}
For each $r\ge 3$ there exist $c,C>0$ and a deterministic $O(n^{r-1})$-time algorithm whose input is an $n$-vertex $r$-uniform hypergraph $G$, a pair of distinct $(r-1)$-tuples $\tpl{u}$ and $\tpl{v}$, a set $S\subset V(G)$ and a $(r-1)$-uniform exposure hypergraph $\cE$ on the same vertex set.
The output of the algorithm is either `Fail' or a spike path of even length $o(\log n)$ in $G$ whose ends are $\tpl{u}$ and $\tpl{v}$ and whose interior vertices are in $S$, and an exposure hypergraph $\cE'\supset\cE$.
We have $e(\cE')\le e(\cE)+ O(|S|^{r-2})$ and all the edges $E(\cE')\setminus E(\cE)$ are contained in $S\cup\tpl{u}\cup\tpl{v}$.

Suppose that $G$ is drawn from the distribution $\cG^{(r)}(n,p)$ with $p\ge C(\log n)^3/n$, that $\cE$ does not contain any edges intersecting both $S$ and $\tpl{u} \cup \tpl{v}$.
If furthermore we have $|S| = Cp^{-1}\log n$ and $|e(\cE[S])| \le c |S|^{r-1}$ then  the algorithm returns `Fail' with probability at most $n^{-5}$.
\end{lemma}
\begin{proof}[Sketch proof]
 We modify the proof of Lemma~\ref{lem:connect} in the following simple ways.
 First, we will maintain fans of spike paths rather than tight paths, and we change Algorithm~\ref{alg:build} line~\ref{line:a} so that the tuple $\tpl{a}$ to be extended is the (unique) one whose extension continues to give us a spike path.
 Note that whenever we have a spike path ending in $\tpl{a}$ and we extend the spike path by adding one vertex $b$ then the end of the new spike path is an $(r-1)$-set whose vertices are contained in $(\tpl{a},b)$ (though in general not the last $r-1$ vertices nor in the same order).
 This is all we need to make our analysis of the fan construction work; it is not necessary to change anything in this part of the proof or the constants.
 Second, when we come to connect fans, we let $L$ be the reverses of the end tuples of $\cF_t(\tpl{u})$ and $L'$ be the end tuples of $\cF_{t}(\tpl{v})$, and (again) look for a tight path connecting a tuple in $L$ to one in $L'$.
 This has no effect on the proof that a connecting path from some member of $L$ to some member of $L'$ exists, and the result is the desired spike path.
 The resulting spike path is of even length as both fans have the same size.
\end{proof}

\section{Proof of the Reservoir Lemma}

\subsection{Idea}

The reservoir path $\Pres$ will consist of absorbing structures (each ``carrying'' one vertex from $R$).
More precisely, these absorbing structures can be seen as small reservoir path with reservoir of cardinality $1$.
Each of these small absorbers consists of a cyclic spike path plus the reservoir vertex, where pairs of spikes are additionally connected with tight paths (cf./ Figure~\ref{figure:Absorber1}).

First we choose the reservoir set $R$ and disjoint sets $U_1$, $U_2$ and $U_3$.
For every vertex in $R$ we will reveal the necessary path segment in $U_1$.
From the endpoints of these path we fan out and also close the backbone structure of the reservoir inside $U_2$.
Finally we use $U_3$ and Lemma~\ref{lem:connect} to get the missing connections in the reservoir structures and connect all structures to one path $\Pres$.
In each step the relevant edges of the exposure graph $\cE$ are solely coming from the same step.

\subsection{Proof}
We arbitrarily fix the reservoir set $R$ of size $2 C p^{-1} \log n$ and disjoint sets $U_1$, $U_2$ and $U_3$ of the same size such that $S=R \cup U_1 \cup U_2 \cup U_3$ is of size $\tfrac{n}{4}$.
First we want to build the absorbing structures for every $a \in R$, which have size roughly $t^2 = o(\log^2 n)$.
There is a sketch of this structure for some $a \in R$ in Figure~\ref{figure:Absorber1}.

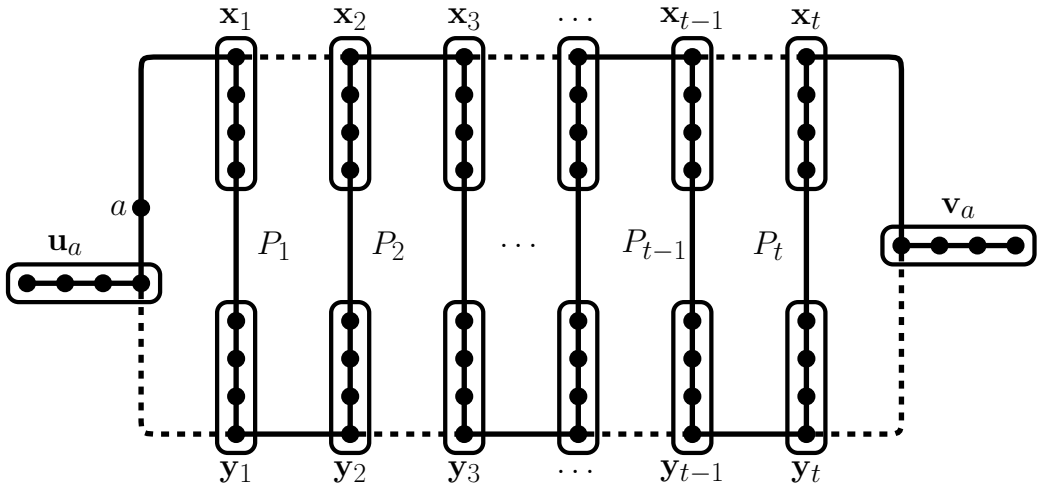
\begin{figure}[htb]
	\begin{center}
		\begin{tikzpicture}[ultra thick,scale=0.5, every node/.style={scale=0.5}]
		
		\tikzset{%
			mycolor/.style={%
				dashed}
		}
		
		
		\node[circle,fill,draw,label=left:{\Huge \bf $a$}] (a) at (-10,1){};		
		\draw[rounded corners] (-9.5,-1.5) rectangle (-13.5,-0.5);	
		\draw[rounded corners] (9.5,0.5) rectangle (13.5,-0.5);	
		
		\foreach \a in {0,1,2,3}
		{
			\node[circle,fill,draw] (u\a) at (-10-\a,-1){};		
			\node[circle,fill,draw] (v\a) at (10+\a,0){};		
		}		
		
		\foreach \a in {0,...,5}
		{
			\draw[rounded corners] ({-8+\a*3},5.5) rectangle ({-7+\a*3},1.5);
			\draw[rounded corners] ({-8+\a*3},-5.5) rectangle ({-7+\a*3},-1.5);
			
			\foreach \b in {0,1,2,3}		
			{
				\node[circle,fill,draw] (x\a\b) at (-7.5+\a*3,5-\b){};
				\node[circle,fill,draw] (y\a\b) at (-7.5+\a*3,-5+\b){};
			}
		}
		
		\node at (-12,0){\Huge \bf $\tpl{u}_a$};
		\node at (11.5,1){\Huge \bf $\tpl{v}_a$};
		\node at (-6.5,0){\Huge \bf $P_1$};
		\node at (-3.5,0){\Huge \bf $P_2$};
		\node at (3.5,0){\Huge \bf $P_{t-1}$};
		\node at (6.5,0){\Huge \bf $P_t$};
		\node at (0,0){\Huge \bf $\dots$};	
		
		\node at (-7.5,6){\Huge \bf $\tpl{x}_1$};			
		\node at (-4.5,6){\Huge \bf $\tpl{x}_2$};		
		\node at (-1.5,6){\Huge \bf $\tpl{x}_3$};	
		\node at (1.5,6){\Huge \bf $\dots$};	
		\node at (4.5,6){\Huge \bf $\tpl{x}_{t-1}$};
		\node at (7.5,6){\Huge \bf $\tpl{x}_t$};
		
		\node at (7.5,-6){\Huge \bf $\tpl{y}_t$};
		\node at (4.5,-6){\Huge \bf $\tpl{y}_{t-1}$};
		\node at (1.5,-6){\Huge \bf $\dots$};	
		\node at (-1.5,-6){\Huge \bf $\tpl{y}_3$};
		\node at (-4.5,-6){\Huge \bf $\tpl{y}_2$};
		\node at (-7.5,-6){\Huge \bf $\tpl{y}_1$};

	\draw[line width=2pt] (u3)--(u2)--(u1)--(u0);
	\draw[line width=2pt] (v3)--(v2)--(v1)--(v0);
	
	\foreach \a in {0,...,5}
	{
		\draw[line width=2pt] (x\a0)--(x\a1)--(x\a2)--(x\a3)--(y\a3)--(y\a2)--(y\a1)--(y\a0);
	}
		
	\draw[mycolor,line width=2pt] (x00)--(x10) (x20)--(x30) (x40)--(x50);
	\draw[mycolor,line width=2pt] (y10)--(y20) (y30)-- (y40);

	\draw[line width=2pt] (y00)--(y10) (y20)--(y30) (y40)--(y50);
	\draw[line width=2pt] (x10)--(x20) (x30)-- (x40);
	
	\draw[mycolor,line width=2pt] (u0)--(-10,-4.5) .. controls (-10,-5) ..(-9.5,-5)--(y00);
	\draw[mycolor,line width=2pt] (v0)--(10,-4.5) .. controls (10,-5) .. (9.5,-5)--(y50);		

	\draw[line width=2pt] (u0) -- (a);
	\draw[line width=2pt] (a)--(-10,4.5) .. controls (-10,5) ..(-9.5,5)--(x00);
	\draw[line width=2pt] (v0)--(10,4.5) .. controls (10,5) .. (9.5,5)--(x50);

		\end{tikzpicture}
	\end{center}
	\caption{Illustration of the absorber for one vertex $a \in R$ and $r=5$ with the path, which contains the vertex $a$.}
	\label{figure:Absorber1}
\end{figure}

So we fix $a \in R$.
We want to construct the following tight path on $2r-1$ vertices containing $a$ in the middle.
The end tuples are $\tpl{x}_1 = (x_1,\dots,x_{r-1})$ and $\tpl{u}_a=(u_1,\dots,u_{r-1})$ and together with $a$ we require that all the edges $\{ x_{r-j},\dots,x_1,a,u_1,\dots,u_{j-1}  \}$ are present for $j=1,\dots,\lfloor r$.
We build this path by first choosing $x_1,\dots,x_{r-2}$ arbitrarily from $U_1$.
Then we expose all edges containing $\{ x_1,\dots,x_{r-2},a \}$ to get $x_{r-1}$.
We continue by exposing all edges containing the set $\{ x_{r-j-1},\dots,x_1,a,u_1,\dots,u_{j-1} \}$ to get $u_j$ for $j=1,\dots, \lfloor r-1$.
The probability that in any of these cases we fail to find a new vertex inside a subset of $U_1$ of size at least $|U_1|/2$ is at most $n^{-5}$ by Chernoff's inequality.
A union bound over all $r$ edges and over all $a \in R$ reveals that with probability at most $n^{-3}$ we fail to construct the small starting graph for any $a$.

Recall that when adding edges, we always expose all edges containing one $(r-1)$-tuple and then add this to $\cE$.
All exposed $(r-1)$-tuples from this step are contained in $U_1 \cup R$ and none of them contains more than one vertex from $R$.
Furthermore we did at most $O(|R|\cdot |U_1|) = O(n^2)$ many steps so far.

Now we want to build the absorbing structure for $a$.
We partition each of $U_2$ and $U_3$ into parts of size $Cp^{-1}\log n$ (plus perhaps a smaller left-over set).
We apply Lemma~\ref{lem:spike} to the $(r-1)$-tuples $\rev{\tpl{x}_1}$ and $\rev{\tpl{u}_a}$ and connect them with a spike path of even length $2t+2$ in some part of $U_2$, with $t = o(\log n)$.
At each step we use a part of $U_2$ in which we have so far built the least spike paths for the application of Lemma~\ref{lem:spike}, which is necessary to control the edges of $\cE$ within this set.
We use $U_2$ as both tuples are contained in $U_1$ and thus we have no problem with edges from $\cE$ intersecting both $U_2$ and the end tuples.
Let the spikes after $\tpl{x}_1$ and $\tpl{u}_a$ be called $\tpl{x}_2,\dots,\tpl{x}_t$ and $\tpl{y}_1,\dots,\tpl{y}_t$ respectively.
The last remaining spike opposite of $\tpl{u}_a$ we call $\tpl{v}_a$.
We apply the tight-path version of Lemma~\ref{lem:connect} to find paths $P_i$ connecting the tuples $\tpl{x}_i$ and $\tpl{y}_i$ for $i=1,\dots,t$ in a part of $U_3$.
Again, we choose a part of $U_3$ which was used for building the least connecting paths so far.
We use parts of $U_3$ for these connections, because all the spikes are contained in $U_1 \cup U_2$ and thus there are no edges of $\cE$ intersecting $U_3$ and the spikes.
This finishes the absorbing structure for $a$.
It has end tuples $\tpl{u}_a$ and $\tpl{v}_a$.

To finish $\Pres$ we enumerate the vertices in $R$ increasingly $a_1,\dots, a_{|R|}$.
Then we use Lemma~\ref{lem:connect} repeatedly, again at each step using a part of $U_3$ which has been used least often previously, to connect the tuples $\tpl{v}_{a_i}$ to $\tpl{u}_{a_{i+1}}$ for $i=1,\dots,|R|-1$ with tight paths.
Thus we have obtained the path $\Pres$ with end tuples $\tpl{u}=\tpl{u}_{a_1}$ and $\tpl{v}=\tpl{v}_{a_{|R|}}$.

The absorbing works in the following way for the structure of a single vertex $a \in R$.
It relies on the fact, that the paths $P_i$ can be traversed in both directions and that we can walk from any spike to its neighbouring spike using a tight path.
The path which uses $a$ (Figure~\ref{figure:Absorber1}) starts with $\tpl{u}_a$, goes through $a$ to $\tpl{x}_1$ and then uses the path $P_1$ to $\tpl{y}_1$.
From there it goes via a tight path to $\tpl{y}_2$ and uses $P_2$ to go back to $\tpl{x}_2$.
Going from $\tpl{x}_i$ via path $P_i$ to $\tpl{y}_i$ and back from $\tpl{y}_{i+1}$ through $P_{i+1}$ to $\tpl{x}_{i+1}$ for $i=2,\dots,t-1$ the path ends up in $\tpl{v}_a$ and uses all vertices.
To avoid $a$ (Figure~\ref{figure:Absorber2}) the path starting in $\tpl{u}_a$ goes immediately to $\tpl{y}_1$, then uses the path $P_1$ to go to $\tpl{x_1}$.
Alternating as above and traversing all the paths $P_i$ in opposite direction we again end up in $\tpl{v}_a$ and used all vertices but $a$.

\begin{figure}[htb]
	\begin{center}
\begin{tikzpicture}[ultra thick,scale=0.5,every node/.style={scale=0.5}]

\tikzset{%
    mycolor/.style={%
        dashed}
}

		
		\node[circle,fill,draw,label=left:{\Huge \bf $a$}] (a) at (-10,1){};		
		\draw[rounded corners] (-9.5,-1.5) rectangle (-13.5,-0.5);	
		\draw[rounded corners] (9.5,0.5) rectangle (13.5,-0.5);	
		
		\foreach \a in {0,1,2,3}
		{
			\node[circle,fill,draw] (u\a) at (-10-\a,-1){};		
			\node[circle,fill,draw] (v\a) at (10+\a,0){};		
		}		
		
		\foreach \a in {0,...,5}
		{
			\draw[rounded corners] ({-8+\a*3},5.5) rectangle ({-7+\a*3},1.5);
			\draw[rounded corners] ({-8+\a*3},-5.5) rectangle ({-7+\a*3},-1.5);
			
			\foreach \b in {0,1,2,3}		
			{
				\node[circle,fill,draw] (x\a\b) at (-7.5+\a*3,5-\b){};
				\node[circle,fill,draw] (y\a\b) at (-7.5+\a*3,-5+\b){};
			}
		}
		
		\node at (-12,0){\Huge \bf $\tpl{u}_a$};
		\node at (11.5,1){\Huge \bf $\tpl{v}_a$};
		\node at (-6.5,0){\Huge \bf $P_1$};
		\node at (-3.5,0){\Huge \bf $P_2$};
		\node at (3.5,0){\Huge \bf $P_{t-1}$};
		\node at (6.5,0){\Huge \bf $P_t$};
		\node at (0,0){\Huge \bf $\dots$};	
		
		\node at (-7.5,6){\Huge \bf $\tpl{x}_1$};			
		\node at (-4.5,6){\Huge \bf $\tpl{x}_2$};		
		\node at (-1.5,6){\Huge \bf $\tpl{x}_3$};	
		\node at (1.5,6){\Huge \bf $\dots$};	
		\node at (4.5,6){\Huge \bf $\tpl{x}_{t-1}$};
		\node at (7.5,6){\Huge \bf $\tpl{x}_t$};
		
		\node at (7.5,-6){\Huge \bf $\tpl{y}_t$};
		\node at (4.5,-6){\Huge \bf $\tpl{y}_{t-1}$};
		\node at (1.5,-6){\Huge \bf $\dots$};	
		\node at (-1.5,-6){\Huge \bf $\tpl{y}_3$};
		\node at (-4.5,-6){\Huge \bf $\tpl{y}_2$};
		\node at (-7.5,-6){\Huge \bf $\tpl{y}_1$};

	\draw[line width=2pt] (u3)--(u2)--(u1)--(u0);
	\draw[line width=2pt] (v3)--(v2)--(v1)--(v0);
	
	\foreach \a in {0,...,5}
	{
		\draw[line width=2pt] (x\a0)--(x\a1)--(x\a2)--(x\a3)--(y\a3)--(y\a2)--(y\a1)--(y\a0);
	}
		
	\draw[line width=2pt] (x00)--(x10) (x20)--(x30) (x40)--(x50);
	\draw[line width=2pt] (y10)--(y20) (y30)-- (y40);

	\draw[mycolor,line width=2pt] (y00)--(y10) (y20)--(y30) (y40)--(y50);
	\draw[mycolor,line width=2pt] (x10)--(x20) (x30)-- (x40);
	
	\draw[line width=2pt] (u0)--(-10,-4.5) .. controls (-10,-5) ..(-9.5,-5)--(y00);
	\draw[line width=2pt] (v0)--(10,-4.5) .. controls (10,-5) .. (9.5,-5)--(y50);		

	\draw[mycolor,line width=2pt] (u0) -- (a);
	\draw[mycolor,line width=2pt] (a)--(-10,4.5) .. controls (-10,5) ..(-9.5,5)--(x00);
	\draw[mycolor,line width=2pt] (v0)--(10,4.5) .. controls (10,5) .. (9.5,5)--(x50);

\end{tikzpicture}
\end{center}
\caption{Illustration of the absorber for one vertex $a \in R$ and $r=5$ with the path, which does not contain the vertex $a$.}
\label{figure:Absorber2}
\end{figure}
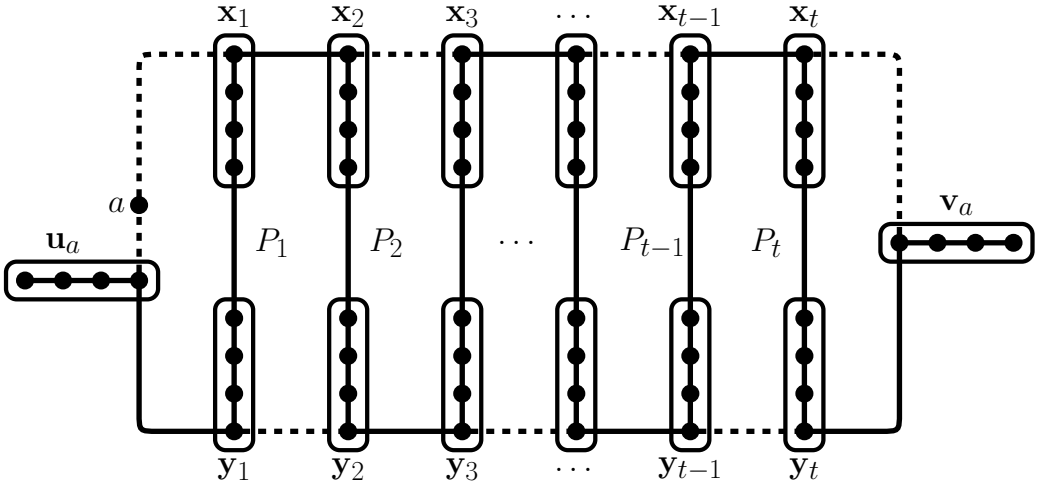

For the proof of the lemma it remains to check that we obtain the right probability and we are indeed able to apply Lemma~\ref{lem:connect} as we described.
It is immediate from the construction, that no edges of $\cE$ are contained in $R \cup \tpl{u} \cup \tpl{v}$.

In total we are performing $|R|$ many connections with spike-paths and $|R|\cdot t + |R| - 1$ many connections with tight-paths.
Thus altogether we have $o\left( p^{-1} {\log^2 n} \right)$ executions of Lemma~\ref{lem:connect} and Lemma~\ref{lem:spike}.
In each application we add $O\big(Cp^{-1}\log n\big)^{r-2}$ edges to $\cE$ in some part of $U_2$ or $U_3$.
Since each part initially contains no edges of $\cE$, provided a given part has been used at most $p^{-1}$ times the total number of edges of $\cE$ in it is $o\big(Cp^{-1}\log n\big)^{r-1}$, and therefore we can apply Lemma~\ref{lem:connect} or~\ref{lem:spike} at least one more time with that part.
Since $|U_2|$ and $|U_3|$ are of size linear in $n$, they each contain $\Omega\big(pn/\log n \big)$ parts.
Thus we can perform in total $\Omega(n/\log n)=\Omega\big(p^{-1} \log^2 n)$ applications of either Lemma~\ref{lem:connect} or Lemma~\ref{lem:spike} before all parts have been used $p^{-1}$ times and thus might acquire too many edges of $\cE$.
Since we do not need to perform that many applications, we conclude that the conditions of each of Lemma~\ref{lem:connect} and Lemma~\ref{lem:spike} are met each time we apply them.

Since the connecting lemma fails with probability at most $n^{-5}$ the construction of this absorber fails with probability at most $n^{-3}$.
In every connection there are at most $O(n^{r-1})$ steps performed and thus we need $o(n^{r-1} p^{-1} \log^2 n) = O(n^r)$ many steps for the construction of the absorber.
\qed

\section{Conclusion}

In this paper we have improved upon the best known algorithms for finding a tight Hamilton cycle in $\Gr(n,p)$:  we provide a deterministic algorithm with runtime $O(n^r)$ which for any edge probability $p\ge C(\log n)^3n^{-1}$ succeeds a.a.s.
While we give an affirmative answer to a question of Dudek and Frieze~\cite{DudFriTight} in this regime, the question remains open for $e/n\le p< C(\log n)^3n^{-1}$ for $r\ge 4$, and $1/n\ll p<C(\log n)^3n^{-1}$ for $r=3$.

Let us now turn our attention to the closely related problem of finding the $r$-th power of a Hamilton cycle in the binomial random graph $\cG(n,p)$, where $r\ge 2$.
While a general result of Riordan~\cite{riordan_2000} already shows that the threshold for $r\ge 3$ is given by $p=\Theta(n^{-1/r})$ (as observed in~\cite{KOPosa}), the threshold for $r=2$ is still open, where the best known upper bound is a polylog-factor away from the first-moment lower bound $n^{-1/2}$~\cite{nenadov2016powers}.

Since the result by Riordan is based on the second moment method it is inherently non-constructive.
By contrast, the proof in~\cite{nenadov2016powers} (for $r\ge 2$) is based on a quasi-polynomial time algorithm which for $p\ge C(\log n)^{8/r}n^{-1/r}$ finds the $r$-th power of an Hamilton a.a.s.\ in $\cG(n,p)$, and which is very similar to their algorithm for finding tight Hamilton cycles in $\Gr(n,p)$.
We think that our ideas are also applicable in this context and would provide an improved algorithm for finding $r$-th powers of Hamilton cycles in $\cG(n,p)$, though we did not check any details.

Finally, it would be interesting to know the average case complexity of determining whether an $n$-vertex $r$-uniform hypergraph with $m$ edges contains a tight Hamilton cycle.
Our results (together with a standard link between the hypergeometric and binomial random hypergraphs) show that if $m\gg n^{r-1}\log^3 n$ then a typical such hypergraph will contain a Hamilton cycle, but the failure probability of our algorithm is not good enough to show that the average case complexity is polynomial time.
For this one would need a more robust algorithm which can tolerate some `errors' at the cost of doing extra computation to determine whether the `error' causes Hamiltonicity to fail or not.


\providecommand{\bysame}{\leavevmode\hbox to3em{\hrulefill}\thinspace}
\providecommand{\MR}{\relax\ifhmode\unskip\space\fi MR }
\providecommand{\MRhref}[2]{%
	\href{http://www.ams.org/mathscinet-getitem?mr=#1}{#2}
}
\providecommand{\href}[2]{#2}

\end{document}